\documentclass[twoside, 12pt]{amsart}

\usepackage[T1]{fontenc}
\usepackage[utf8]{inputenc}
\usepackage[english]{babel}
\usepackage{indentfirst}
\usepackage{amsmath, amssymb, amscd, amsthm, amsfonts}
\usepackage{amssymb, amsfonts}
\usepackage{amsthm, amscd}

\usepackage{inputenc}
\usepackage{mathtools}
\usepackage{hyperref}
\usepackage{verbatim}
\usepackage{color}
\usepackage{enumerate}
\usepackage{mathrsfs}

\theoremstyle{definition}
\newtheorem{definition}{Definition}

\theoremstyle{plain}
\newtheorem{theorem}[definition]{Theorem}

\theoremstyle{plain}

\theoremstyle{plain}

\theoremstyle{plain}

\theoremstyle{definition}
\newtheorem{remark}[definition]{Remark}
\theoremstyle{plain}
\newtheorem{lemma}[definition]{Lemma}
\theoremstyle{plain}

\theoremstyle{plain}
\newtheorem{corollary}[definition]{Corollary}
\theoremstyle{definition}

\theoremstyle{definition}

\newcommand{\R}{\mathbb{R}}
\newcommand{\Reals}{\mathbb{R}}
\newcommand{\N}{\mathbb{N}}

\newcommand{\lef}{\left(}

\DeclareMathOperator{\conv}{Conv}

\newcommand{\beq}{\begin{equation}}
\newcommand{\eeq}{\end{equation}}

\DeclareMathOperator{\dist}{dist}

\begin{document} 

\date{\today}
\title{Ancient Mean Curvature Flows\\and their Spacetime Tracks}

\author{Francesco Chini\\}
\author{\\Niels Martin M\o{}ller}
\address{Francesco Chini, Department of Mathematical Sciences, Copenhagen University.}
\email{chini@math.ku.dk}

\address{Niels Martin M\o{}ller, Department of Mathematical Sciences, Copenhagen University.}
\email{nmoller@math.ku.dk}
\thanks{Francesco Chini was partially supported by the Villum Foundation's QMATH Centre.  Both authors were partially supported by Niels Martin M\o{}ller's Sapere Aude grant from The Independent Research Fund Denmark (Danish Ministry).}

\begin{abstract}
We study properly immersed ancient solutions of the codimension one mean curvature flow in $n$-dimensional Euclidean space, and classify the convex hulls of the subsets of space reached by any such flow.

In particular, it follows that any compact convex ancient mean curvature flow can only have a slab, a halfspace or all of space as the closure of its set of reach.

The proof proceeds via a bi-halfspace theorem (also known as a wedge theorem) for ancient solutions derived from a parabolic Omori-Yau maximum principle for ancient mean curvature flows.

\end{abstract}
\maketitle



\section{Introduction}

Ancient mean curvature flows show up naturally in the study of singularities, as tangent flows from blow-up analysis (for a basic discussion, see e.g. Chapter 4 in \cite{mantegazza}). Especially in recent years they have gained much attention and some partial classifications are now available.

Convex ancient solutions arise in the case of mean convex mean curvature flows \cite{white00}, \cite{HS99a}, \cite{HS99b} and have also been investigated by \cite{wang}.

In the case of the curve shortening flow,  Daskalopoulos, Hamilton and Sesum  provided a complete classification  \cite{dhs_2010} of closed convex embedded ancient curve shortening flows. Note that there also exist nonconvex examples \cite{ay18}.

In \cite{huisken_sinestrari}, Huisken and Sinestrari again studied ancient mean curvature flow under several natural curvature assumptions, namely convexity and $k$-convexity, and provided some characterizations of the shrinking sphere, assuming convexity. 

Haslhofer and Hershkovits \cite{haslhofer_hershkovits} proved the existence of an ancient oval in dimensions $n>1$, as conjectured by Angenent in \cite{angenent13}, building on an idea of White \cite{white03}. Recently Angenent, Daskalopoulos and Sesum \cite{ads_18}  proved the uniqueness of this ancient mean curvature flow under some assumptions, based on their previous work \cite{ads15} (using the barriers from \cite{steve-niels}). More precisely they prove that an ancient mean curvature flow which is compact, smooth, noncollapsed, not self-similar and uniformly 2-convex must be the solution constructed in \cite{white03}-\cite{angenent13}-\cite{haslhofer_hershkovits}.

In \cite{brendle_choi} Brendle and Choi classified convex, noncompact, noncollapsed ancient flows in $\R^3$, proving that they agree (up to isometry and up to scaling) with the self-translating bowl soliton. Recently, in \cite{brendle_choi_2018}, they extended their result to higher dimensions, under the extra assumption of uniform 2-convexity.

In \cite{choi_haslhofer_hershkovits}  Choi,  Haslhofer,  Hershkovits classified all 2-dimensional ancient mean curvature flows with low entropy in $\R^3$.

In this paper we generalize some of the results from \cite{chini_moeller}, which classified the projected convex hulls of all proper self-translaters. This establishes the following string of generalizations: The below time-dependent Theorem \ref{theorem_classify} for ancient flows implies the time-independent self-translating hypersurfaces case \cite{chini_moeller}, which again implies the minimal hypersurface case \cite{hoffman_meeks}, which finally implies the Euclidean case of conically bounded minimal surfaces \cite{omori}.

Note that in Theorem \ref{theorem_ancient} and Theorem \ref{theorem_classify} below, we do not have any curvature or non-collapsing nor entropy assumptions. Here we only need to assume the flows to be properly immersed. 

We expect these results to be useful in the future investigation of ancient solutions, both for problems of classification and construction of examples, and hence to the investigation of the set of possible singularities in the mean curvature flow.

Finally, also in regularity questions for mean curvature flow with boundary, bi-halfspace theorems (a.k.a. wedge theorems) are useful: In January 2019, Brian White posted a paper \cite{white19} with a result on boundary regularity (announced some time ago, e.g. in \cite{white09}, see also \cite{stone}), proved there using a new wedge theorem for self-shrinking Brakke flows. It would be interesting to understand its relation to our smooth results in \cite{chini_moeller} and in Theorem \ref{theorem_ancient} and Theorem \ref{theorem_classify} below.


%

\section{Preliminaries}

\begin{definition}
Let $M^n$ be a smooth, connected $n$-dimensional manifold without boundary and let $I \subseteq \R$ be a (time) interval. A mean curvature flow is a smooth map $F \colon M \times I \to \R^{n+1}$ such that $F_t : M \to \R^{n+1}$ is an immersion for every $t \in I$, where $F_t(x) \coloneqq F(x, t)$, and $F$ satisfies the following equation
\begin{equation}\label{mcf_equation}
\frac{\partial F}{\partial t}= \overrightarrow{H}.
\end{equation}
The mean curvature flow is said to be an \emph{ancient}, \emph{immortal} or \emph{eternal} solution,  if respectively after a time translation $I = (-\infty, 0)$, $(0, \infty)$ or $\R$.
\end{definition}

In what follows, we will denote by $M_t$ the manifold $M$ endowed with the pullback metric induced by $F_t$, i.e.
$$
M_t \coloneqq \left(M, F_t^* \langle \cdot, \cdot \rangle_{\R^{n+1}} \right).
$$
The Levi-Civita connection and the Laplacian on $M_t$ will be denoted by $\nabla^{M_t}$ and $\Delta^{M_t}$ respectively.

Moreover, we will always consider proper mean curvature flows, meaning that for every $t \in I$ the map $F_t: M\to\Reals^{n+1}$ is a proper immersion. We remind the reader that properly immersed hypersurfaces are geodesically complete w.r.t. the induced Riemannian metric (by the Heine-Borel property and Hopf-Rinow). As always, most of our results fail without the properness assumption, see e.g. the examples in \cite{nadirashvili} of minimal surfaces non-properly immersed into ambient balls.

\section{Main results}

\begin{lemma}[Omori-Yau Maximum Principle for Ancient MCFs]\label{omori_yau}

Let  $F \colon M \times (-\infty, 0 ) \to \R^{n+1}$ be a proper ancient mean curvature flow. 
Let $f \colon M \times (-\infty, 0) \to \R$ be a bounded and twice differentiable function.

Then there is a sequence of points $(x_i, t_i) \in M \times (-\infty, 0)$ such that 
\begin{enumerate}[(i)]
\item $\lim_{i \rightarrow \infty} f(x_i, t_i) = \sup_{M \times   (-\infty, 0)} f$,
\item $\lim_{i \rightarrow \infty} |\nabla^{M_{t_i}} f(x_i, t_i)| = 0$,
\item $\liminf_{i \rightarrow \infty} \left( \frac{\partial}{\partial t} - \Delta^{M_{t_i}} \right) f(x_i, t_i) \ge 0$.
\end{enumerate}
\end{lemma}

\begin{theorem}[Wedge Theorem for Ancient Mean Curvature Flows]\label{theorem_ancient}
Let $H_1 $ and $H_2$ be two halfspaces of $\R^{n+1}$ such that the hyperplanes $P_1 \coloneqq \partial H_1 $ and $P_2 \coloneqq \partial H_2$ are not parallel. 

Then $H_1 \cap H_2$ does not contain any proper ancient mean curvature flow. More precisely, there does not exist any proper ancient mean curvature flow $F \colon M \times (-\infty , 0) \to \R^{n+1}$ such that $F_t(M) \subseteq H_1 \cap H_2$ for all times $t \in (-\infty, 0)$.
\end{theorem}

\begin{remark}\label{remark_eternal}
Theorem \ref{theorem_ancient} holds in particular for proper eternal mean curvature flows, i.e. $I=\Reals$. There are two particularly important subclasses of eternal mean curvature flows: self-translating solitons and minimal hypersurfaces. Therefore it generalizes Theorem~1 contained in the paper \cite{chini_moeller} by the present authors (see also the corollaries and discussion there), as well as classical theorems by Omori \cite{omori} (in the Euclidean case) and Hoffman-Meeks \cite{hoffman_meeks} (in the case without boundary). A third type of ancient solutions are the self-shrinking solitons (where there are many examples, see e.g. \cite{kkm}-\cite{torusphere}), which even obey a halfspace theorem \cite{CE16} (proved by Cavalcante-Espinar using the barriers from \cite{steve-niels}).

As in \cite{chini_moeller}, it is interesting to ask which of the cases can actually occur in Theorem \ref{theorem_classify}. For each case we have respectively: Flat planes, reaper cylinders (plus ``Angenent ovals'' \cite{ang92} and ``ancient pancakes'' \cite{BLT17}), which give slabs. No examples (to our knowledge) of halfspaces. Spheres, cylinders and the bowl soliton for all of $\R^{n+1}$. Note of course that by \cite{CE16}, self-shrinkers cannot provide examples in the halfspace case (see also the discussion in \cite{chini_moeller}).

Note also that Theorem \ref{theorem_ancient} is the ``bi-halfspace'' result we can expect for ancient mean curvature flows. In fact a ``halfspace theorem'' version would be false. There are several counterexamples: for instance planes and grim reaper cylinders. Also, a ``halfspace'' statement would be false, even for those ancient mean curvature flows all of whose time-slices are compact: A counterexample to this is given by the so-called \emph{ancient pancake} \cite{BLT17} (or for $n=1$, Angenent's ovals \cite{ang92}) which is contained for all its evolution in a slab between two parallel hyperplanes, and thus no general halfspace theorem could hold for all ancient solutions.  

We remark that the statement of Theorem \ref{theorem_ancient} is false for general immortal mean curvature flows, i.e. $I=(0,\infty)$. In fact there are self-expanding mean curvature flows such that they are contained for their entire evolution in the intersection of two halfspaces with nonparallel boundaries (see f.ex. \cite{schnuschu}). Note that f.ex. Lemma \ref{Libound} as stated, and hence the proof of Theorem \ref{theorem_ancient}, would have failed if we had instead taken $I=(0,\infty)$.

\end{remark}

\begin{theorem}[Classification of Sets of Reach of Ancient Flows]\label{theorem_classify}
Consider a proper ancient mean curvature flow. Let
\[
\mathcal{R}:= \bigcup_{t\in(-\infty,0)}F_t(M)\subseteq \R^{n+1}
\]
denote its set of reach.
Then the convex hull $\conv(\mathcal{R})$ is either a hyperplane, a slab, a halfspace or all of $\R^{n+1}$.
\end{theorem}

In the next corollary, we keep track of the time coordinate to get a spacetime track version, which is of course equivalent to Theorem \ref{theorem_classify}.

\begin{corollary}[Spacetime Tracks of Ancient Flows]\label{ST}
Consider for a proper ancient mean curvature flow its spacetime track $\mathcal{ST}$
\[
\mathcal{ST}:= \bigcup_{t\in(-\infty,0)}\{t\}\times F_t(M)\subseteq \R\times\R^{n+1}.
\]
Then $\conv(\pi_2(\mathcal{ST}))$ is either a hyperplane, a slab, a halfspace or all of $\R^{n+1}$, where $\pi_2$ denotes the projection to the $\R^{n+1}$-factor.
\end{corollary}

In the next corollary, the set $\mathcal{R} \cup \{p_\infty\}$ can also be thought of simply as the closure of the set of reach.

\begin{corollary}[Sets of Reach of Compact Convex Ancient Flows]\label{cor_convex}
Consider any compact convex ancient mean curvature flow in $\R^{n+1}$, which at time $0$ becomes extinct at a point $p_\infty \in \R^{n+1}$.

Then $\mathcal{R} \cup \{p_\infty\}$ (the set of points reached, with the singular point added in) is either a slab, a halfspace or all of $\R^{n+1}$. 
\end{corollary}

\begin{remark}
Note that Corollary \ref{cor_convex} is in agreement with  Corollary 6.3 in \cite{wang} where blow-downs $(-t)^{-\frac{1}{2}} F_t(M)$ as $t \rightarrow -\infty$  for convex ancient solutions were classified: any $\mathbb{S}^k \times \R^{n-k}$, with  $k =1, \dots, n $ or  multiplicity two hyperplanes. 
It is not clear to us whether the halfspace case of Corollary \ref{cor_convex} could   occur. For instance in the convex case it does not happen in the curve shortening flow (i.e. the case $n=1$)  because of the classification for closed curves in \cite{dhs_2010}, which shows that the only possible sets of reach are strips and $\R^2$.

Moreover Wang (Corollary 6.1 \cite{wang}) showed that the set of reach of a (not necessarily compact) convex ancient mean curvature flow arising as a limit flow of a mean convex flow, is the entire $\R^{n+1}$. Note that the set of reach there is taken over the whole maximal time interval, which might be $(-\infty, \infty)$.

\end{remark}

\section{Proofs}

The proof of Lemma \ref{omori_yau} is based on the Omori-Yau maximum principle (tracing its roots back to \cite{omori}--\cite{cheng-yau}) proven by Ma in \cite{john_ma}. The main difference is that here we are interested in ancient mean curvature flows and thus our time interval is not finite, which complicates slightly (but essentially) the proof. On the other hand, because of the applications we have in mind, we focus on the case where the ambient manifold is Euclidean space $\R^{n+1}$ and the codimension is 1, and the argument we give here is self-contained.

\begin{proof}[Proof of Lemma \ref{omori_yau}]

Let $(\bar{x}_i, \bar{t}_i)$ be a sequence in $M \times (-\infty, 0)$ such that
\begin{equation}\label{defbars}
\lim_{i \rightarrow \infty} f(\bar{x}_i, \bar{t}_i) = \sup_{M \times (-\infty, 0)}f.
\end{equation}

Consider the function $r \colon \R^{n+1} \to \R$ defined as $r(y) \coloneqq \|y\|$. This defines a function $\varrho$ on $M \times (-\infty, 0)$ by $\varrho(x,t) \coloneqq r(F(x,t))$.

Now let $(\varepsilon_i)_{i \in \N}$ be the sequence of positive numbers (well-defined even if $\varrho(\bar{x}_i, \bar{t}_i)=0$)
\beq
0<\varepsilon_i:=\min\left(\frac{1}{i},\frac{1}{i}\frac{1}{\varrho(\bar{x}_i, \bar{t}_i) ^2}\right)<\infty.
\eeq

Note that $\lim_{i\to\infty}\varepsilon_i = 0$ and that for every $i \in \N$
\beq\label{epsrho}
\varepsilon_i \varrho(\bar{x}_i, \bar{t}_i) ^2 \le \frac{1}{i}.
\eeq

Let us now for $i=1,2,\ldots$ define $f_i: M\times (-\infty,0)\to \Reals$ by
\beq\label{fidef}
f_{i} (x, t) \coloneqq f(x,t) - \varepsilon_i \left(\varrho(x, t)\right)^2.
\eeq
Note that each $f_i $ is bounded from above by $\sup_{M \times   (-\infty, 0)} f < \infty$. 

\underline{Claim:} Fix a time $t \in (-\infty, 0)$ and fix $i\in \mathbb{N}$. Then there exists a point $x_t^i\in M$ where the function $f_i(\cdot, t) $  attains its supremum over $M$. Furthermore, this is locally uniform in the sense that considering $\tau$ near $t$, all the points $x_{\tau}^i\in M$ can be chosen from a fixed compact subset $K\subseteq M$ (with $K=K_t$ possibly dependent on $t$ and on the proximity of $\tau$ to $t$).

If $M$ is compact, then the claim is trivial. If $M$ is not compact, it follows from the crucial properness assumption. In fact, let $R > 0$ be large enough so that $F_t(M) \cap B_R \ne \emptyset$, where $B_R$ is the ambient open ball of radius $R>0$ in $\R^{n+1}$ centered at $0$. Since $f$ is bounded on $M \times (-\infty, 0)$, we can choose $S>R>0$ so that
\beq\label{minusepses}
 \sup_{M \times (-\infty, 0)} f - \varepsilon_i S^2 <\inf_{M \times (-\infty, 0)} f - \varepsilon_i R^2
\eeq

Equation (\ref{fidef}) now shows that for points $p \in M \setminus F^{-1}_t(B_S)$ (which  is nonempty, because from the properness of $F_t$ and noncompactness of $M$ follows that $F_t(M)$ cannot be contained in any finite radius ambient ball) holds: 
\beq
f_i(p, t) \leq f(p, t) - \varepsilon_i S^2 \le \sup_{M \times (-\infty, 0)} f - \varepsilon_i S^2.
\eeq

Therefore taking the supremum over $p\in M \setminus F^{-1}_t(B_S)$  yields
\begin{align}
\sup_{M \setminus F^{-1}_t (B_S)} f_i(\cdot, t) \leq \sup_{M \times (-\infty, 0)} f - \varepsilon_i S^2 < \inf_{M \times (-\infty, 0)} f - \varepsilon_i R^2, 
\end{align}
using (\ref{minusepses}). Thus, finally, using $f-\varepsilon_i R^2\leq f_i$ on $F^{-1}_t(B_R)$:
\beq\label{sup-bit}
\sup_{M \setminus F^{-1}_t (B_S)} f_i(\cdot, t) < \inf_{F^{-1}_t(B_R)} f_i(\cdot,t)
\eeq

From continuity of $F$, we also have that there exists $\delta >0$ such that, for every time $\tau \in (t - \delta, t+\delta)$, we have $F_\tau(M) \cap B_R \ne \emptyset$ and thus  \eqref{sup-bit} still holds. Properness of the flow implies that each $F^{-1}_\tau(\bar{B}_S) \subseteq M$ is compact, therefore we have for every $\tau \in (t - \delta, t + \delta)$:
\begin{equation}\label{eq_max_prop}
\sup_{M} f_i(\cdot, \tau) = \max_{F^{-1}(\bar{B}_S)} f_i(\cdot,\tau ).
\end{equation}

It only remains to find a uniform compact $K \subseteq M$ as claimed. This is also guaranteed by the properness of the immersion at each time, as follows.

Since  $F^{-1}_t (\bar{B}_{S})$ is compact, we can choose $K$ to be any larger compact set such that $F^{-1}_t (\bar{B}_{S}) \subseteq K^\circ\subseteq M$, where $K^\circ$ denotes the interior of $K$. Consider the closed set $C \coloneqq M \setminus K^\circ$. The two sets $F_t(C)$ and $B_{S}$ are disjoint. Together with compactness of $\bar{B}_{S}$, and the assumption that $F_t$ is proper (and hence a closed map), which ensures closedness of $F_t(C)$, it then implies $\dist_{\R^{n+1}}(F_t(C), \bar{B}_{S}) > 0$.

From the triangle inequality follows that $\dist_{\R^{n+1}}(F_\tau(C), B_{S})$ is continuous in $\tau\in  (t-\delta, t+\delta)$. Therefore after possibly taking $\delta >0$ smaller holds $\dist_{\R^{n+1}}(F_\tau(C), B_{S}) > 0$ for all $\tau\in (t-\delta, t+\delta)$. But then as claimed $F_\tau^{-1}(B_{S})\subseteq K$ and finally, using (\ref{eq_max_prop}) we finish the proof of the final part of the claim:
\beq\label{mantegazza_K}
\forall \tau\in I:\quad \sup_{M}f_i(\cdot, \tau) =\max_{K} f_i(\cdot, \tau).
\eeq

For any time $t\in (-\infty, 0)$ and $i\in\mathbb{N}$, let us use the claim and denote
\beq\label{slicemax}
L_i(t):= \max_{x\in M} f_i(x, t)=f_i(x^i_t, t),
\eeq
for some $x^i_t\in K_t$.

Note that the function $(-\infty, 0) \ni t \mapsto L_i(t)$ is bounded from above by $L:=\sup_{M \times   (-\infty, 0)} f$. 

By Hamilton's Trick \cite{hamiltonstrick} (see f.ex. Lemma 2.1.3. in \cite{mantegazza}, which is where we use the uniformicity property (\ref{mantegazza_K}) and the $K_t$ in the claim), each $L_{i}$ is a locally Lipschitz function of $t$ and therefore continuous. The function $L_i$ is also differentiable almost everywhere, and at any of its differentiability times $ t \in (-\infty, 0)$ we have as usual
$$
\frac{dL_i}{d t} (t) = \frac{\partial f_i }{\partial t} (x_t^i, t).
$$

This implies (using Lemma \ref{lemma_lipschitz} in the Appendix) that there is a time $t_i \in (-\infty, 0)$  such that $\frac{d L_i}{dt}(t_i)$ exists and with $x_i \coloneqq x^i_{t_i}$ there holds
\begin{equation}\label{Libound}
\frac{\partial f_i}{\partial t}  (x_i, t_i) = \frac{d L_i}{dt}(t_i) \geq - 2\varepsilon_i,
\end{equation}
and also
\begin{equation}
\bigg| L_i(t_i) - \sup_{(-\infty,0)} L_i\bigg| < \varepsilon_i,
\end{equation}
or in other words
\begin{equation}\label{sup_L_i}
\bigg| f_i(x_i,t_i) - \sup_{M\times (-\infty,0)} f_i\bigg| < \varepsilon_i.
\end{equation}

We have from (\ref{slicemax}) that $\Delta^{M_{t_i}} f_i(x_i,t_i)\leq 0$, therefore with \eqref{Libound}
\begin{equation}\label{eq_max}
\left(\frac{\partial}{\partial t} - \Delta^{M_{t_i}} \right) f_{i}(x_i, t_i) \ge - 2\varepsilon_i.
\end{equation}

From standard computations and \eqref{mcf_equation} we also have (as the only step in the proof where we use the mean curvature flow equation) the following
\begin{equation}\label{eq_dist}
\left(\frac{\partial}{\partial t} - \Delta^{M_{t}} \right) \left(\varrho (x, t)\right)^2  = -2n,
\end{equation}
for every $(x, t) \in M \times (-\infty, 0)$.

Combining \eqref{eq_max} and \eqref{eq_dist}, we get with (\ref{fidef})

\begin{equation}
\left(\frac{\partial}{\partial t} - \Delta^{M_{t_i}} \right) f (x_i, t_i) \ge -2(n+1)\varepsilon_i.
\end{equation}
This shows Part \emph{(iii)} of the Lemma.
Let us now check that also \emph{(i)} and \emph{(ii)} hold. 

From \eqref{fidef} and \eqref{sup_L_i}, and since $L_i(t_i)=f_i(x_i, t_i) = \max_{M} f_i(\cdot, t_i)$, we have (see \eqref{defbars} for the definition of $(\bar{x}_i, \bar{t}_i)$)

\beq\label{inequality_f_f_i}
\begin{split}
f(x_i, t_i) &\ge f_i(x_i, t_i) > \sup_{M \times (-\infty, 0)} f_i - \varepsilon_i\\ &\ge f_i(\bar{x}_i, \bar{t}_i) - \varepsilon_i \geq f(\bar{x}_i, \bar{t}_i) - \frac{1}{i} - \varepsilon_i,
\end{split}
\eeq
where the final inequality made use of \eqref{epsrho}. This shows Part \emph{(i)}, by taking the limit for $i \rightarrow \infty$ in the string of inequalities \eqref{inequality_f_f_i}. 

Let us now show Part \emph{(ii)}. Observe that $\nabla^{M_{t_i}}f_i(x_i, t_i)= 0$, since $x_i\in M$ is a maximum point for $f_i(\cdot, t_i)$. Thus we have
$$
\nabla^{M_{t_i}} f(x_i, t_i) = 2 \varepsilon_i \varrho(x_i, t_i) \nabla^{M_{t_i}}\varrho(x_i, t_i) = 2 \varepsilon_i \varrho(x_i, t_i) \left( \left(\nabla^{\R^{n+1}}  r\right)(F(x_i, t_i ) \right)^\top.
$$ 

Noting
\beq
\left\|\left( \left(\nabla^{\R^{n+1}}  r\right)(F(x_i, t_i ) \right)^\top\right\| \le  \left\| \left(\nabla^{\R^{n+1}}  r\right)(F(x_i, t_i ) \right\| \le 1,
\eeq
it is enough to show that $ \varepsilon_i \varrho(x_i, t_i) \xrightarrow[ i \rightarrow \infty]{} 0$ (note that \eqref{epsrho} concerns $(\bar{x}_i, \bar{t}_i)$). But from \eqref{inequality_f_f_i}, we have $f_i(x_i, t_i) > f(\bar{x}_i, \bar{t}_i) - \frac{1}{i} - \varepsilon_i$, so that
$$
\varepsilon_i \varrho(x_i, t_i)^2 = f(x_i, t_i) -f_i(x_i, t_i) < \sup_{M \times (-\infty, 0)}f - f(\bar{x}_i, \bar{t}_i) + \frac{1}{i} + \varepsilon_i.
$$
Therefore $\sqrt{\varepsilon_i} \varrho(x_i, t_i) \xrightarrow[i \rightarrow \infty]{} 0$ and therefore by (\ref{defbars}), finally

\[
\varepsilon_i \varrho(x_i, t_i) \xrightarrow[ i \rightarrow \infty]{} 0.
\]
\end{proof}


\begin{proof}[Proof of Theorem \ref{theorem_ancient}]
The rest of the proof is very similar to that of Theorem 1 in \cite{chini_moeller}, in the case of self-translating solitons without boundary, the proof of which was in turn inspired by an idea for 2-dimensional minimal surfaces in $\Reals^3$ by Borb\'e{}ly in \cite{bo11}. In particular we are going to apply Lemma \ref{omori_yau} to a function $f$ which is constructed exactly in the same way as in \cite{chini_moeller}.

Let $H_1, H_2 \subseteq \R^{n+1}$ be two halfspaces such that $P_1 \coloneqq \partial H_1$ and $ P_2 \coloneqq \partial H_2$ are not parallel. Let us assume by contradiction that there exists a proper ancient mean curvature flow $F \colon M \times (-\infty, 0) \to \R^{n+1}$ such that $F_t(M) \subseteq H_1 \cap H_2$ for every $t \in (-\infty, 0)$.

We can assume without loss of generality that $0 \in P_1 \cap P_2$. 
Let $w_1, w_2 \in \mathbb{S}^{n}$ such that $H_i = \{x \in \R^{n+1} \colon \langle x, w_i\rangle \ge 0\}$.

For $R>0$, let $\mathcal{L}_R \subseteq \R^{n+1}$ be the $(n-1)$-dimensional affine subspace obtained by translating $P_1 \cap P_2$ in the direction of $w_1 + w_2$ and such that the boundary of the solid cylinder 
$$
\mathcal{D}_R \coloneqq \{x \in \R^{n+1} \colon \dist(x, \mathcal{L}_R) \le R\}
$$ 
is tangent to $P_1$ and $P_2$. 

Let $d_R \colon \R^{n+1} \to \R$ denote the distance function from $\mathcal{L}_R$, i.e. $d_R(x) \coloneqq \dist(x, \mathcal{L}_R)$. 
Observe that $(H_1 \cap H_2) \setminus \mathcal{D}_R$ consists of two connected components. Let $\mathcal{V}_R$ be the one where $d_R$ is bounded. Let us choose $R>0$ large enough such that there exists $t \in (-\infty, 0)$ such that $F_t(M) \cap \mathcal{V}_R \ne \emptyset$.

Let us now define a function $f \colon M \times (-\infty, 0) \to \R$ as follows
\begin{equation}\label{definition_f}
f(x,t) \coloneqq 
\begin{cases}
d_R(F(x, t)) \qquad &\text{ if } F(x,t) \in \mathcal{V}_R \\
R \qquad &\text{ otherwise.}
\end{cases}
\end{equation}

Observe that by construction $f$ is continuous and bounded. Since we have chosen $R>0$ in such a way that $F(x, t) \in \mathcal{V}_R$ for some $(x,t) \in M \times (-\infty, 0)$, we have that 
\beq\label{sup_of_f}
0< R < \sup_{F^{-1}(\mathcal{V}_R)} f = \sup_{M \times (-\infty, 0)} f <\infty.
\eeq

We want to apply the Omori-Yau maximum principle in the form of Lemma \ref{omori_yau} to $f$. Note that $f$ is smooth on the interior of $F^{-1}(\mathcal{V}_R)$ and, because of \eqref{sup_of_f}, this is actually enough in order to apply Lemma \ref{omori_yau}.

By standard computations, see e.g. \cite{chini_moeller}, one can check that on $F^{-1}(\mathcal{V}_R)$ we have
\begin{equation}\label{equation_heat_operator_f}
\lef \frac{\partial}{\partial t} - \Delta^{M_t} \right) f = - \frac{1 - \|\nabla^{M_t} f \|^2}{d_R}.
\end{equation}

Let $(x_i, t_i) \in M \times (-\infty, 0)$ be an Omori-Yau sequence given by Lemma \ref{omori_yau}. From \eqref{equation_heat_operator_f}, Part \emph{(ii)} of Lemma \ref{omori_yau} and \eqref{sup_of_f}, we have that the function $f$ eventually becomes strictly subcaloric at points in the sequence:
$$
\lim_{i \rightarrow \infty} \lef \frac{\partial}{\partial t} - \Delta^{M_{t_i}} \right) f(x_i, t_i) = - \frac{1}{\sup_{M \times (-\infty, 0)} f} < 0.
$$

On the other hand, this is in contradiction with Part \emph{(iii)} of Lemma \ref{omori_yau}, which concludes the proof.
\end{proof}

\begin{proof}[Proof of Theorem \ref{theorem_classify}]
This proof proceeds quite like in the case of minimal surfaces \cite{hoffman_meeks} and self-translaters \cite{chini_moeller}:
\[
\conv(\mathcal{R}) = \bigcap \{H \subseteq \R^{n+1} \colon H \text{ is a halfspace s.t. }\mathcal{R}\subseteq H\},
\]
the intersection of all halfspaces containing the set of reach. If any such two halfspaces $H_1$ and $H_2$ were not parallel, we would conclude that for all times $t\in(-\infty,0)$ the flow is contained in a non-halfspace wedge, $F_t(M)\subseteq H_1\cap H_2$, violating Theorem \ref{theorem_ancient}. Hence the conclusion follows.
\end{proof}

\begin{proof}[Proof of Corollary \ref{cor_convex}]
Let us first remind the reader that by a convex hypersurface $\Sigma^n\subseteq\Reals^{n+1}$ we mean one where all principal curvatures $\kappa_i>0, i=1,\ldots,n$, and that by a theorem of Sacksteder \cite{sacksteder}, this implies that $\Sigma = \partial \Omega$, for some strictly convex domain in $\Reals^{n+1}$. Knowing this we immediately rule out the ``flat plane minus one point'' as a possible set of reach.

Let now $F \colon M \times (-\infty, 0)\to\R^{n+1} $ be a mean curvature flow as in the statement.  Following Huisken \cite{Hu84}, the flow will become extinct at a ``round point'' $p_\infty \in \R^{n+1}$ at time 0. Let $\Omega_t$ be the bounded convex body such that $\Sigma_t = \partial \Omega_t$. We have that the flow sweeps out the interior of each $\Omega_t$. These facts easily imply that adding the singular point to $\mathcal{R}$ we get a convex set: Namely, suppose that $p_1,p_2\in \mathcal{R}$ are given. Then there exist $t_1,t_2\in(-\infty,0)$ so that $p_i\in F_{t_i}(M)$, and with $t_0:=\min(t_1,t_2)$ we have, by the monotonicity of the domains $\Omega_t$, that $p_1,p_2\in \Omega_{t_0}$. Also, considering the line segment between $p_1$ and $p_2$ it is, by convexity of $\Omega_{t_0}$, contained in $\Omega_{t_0}$. Hence by the property that the flow sweeps the interior of $\Omega_{t_0}$, the line segment is contained in $\mathcal{R} \cup \{p_\infty\}$. The case where one $p_i = p_\infty$ follows similarly (or by continuity).

Thus, having shown that $\conv(\mathcal{R}) = \mathcal{R} \cup \{p_\infty\}$ under these extra assumptions, we apply Theorem \ref{theorem_ancient} to finish the proof of Corollary \ref{cor_convex}.
\end{proof}

\newpage
\section{Appendix}

In this section, we state and prove the following elementary lemma, needed in the proofs in the paper's main sections:

\begin{lemma}\label{lemma_lipschitz}
Let $L \colon (-\infty, 0) \to \R$ be a locally Lipschitz function bounded from above. 

Then for every $\varepsilon>0$ there exists some $t_0 \in (-\infty, 0)$ such that $L$ is differentiable at $t_0$ and satisfies the following
\begin{enumerate}[(i)]
\item $L'(t_0) \ge - \varepsilon$,
\item[]
\item $ L(t_0) > \sup_{(-\infty, 0)} L  - \varepsilon.$
\end{enumerate}
\end{lemma}

\begin{proof}[Proof of the Lemma]
Recall that Lipschitz continuity implies absolute continuity.  Let us fix $\varepsilon >0$. 
Let us first assume that there exists $t_{0} \in (-\infty, 0)$ such that 
\beq
L(t_{0}) = \sup_{(-\infty, 0)} L.
\eeq
If $L$ is differentiable at $t_1$, then we are done. Let us assume it is not. Let $\delta>0$ be such that $|L(t) - L(t_1)| < \varepsilon$ for any $|t -t_1| < \delta $. 
Then
\beq
\int_{t_1 - \delta}^{t_1} L'(t) \, dt = L(t_1) - L(t_1 - \delta) \ge 0.
\eeq
Therefore there exists $t_0 \in (t_1 - \delta, t_1)$ such that $L$ is differentiable at $t_0$ and such that $L'(t_0) \ge 0$. Moreover $|L(t_0) - L(t_1)| < \varepsilon$. 

Let us  now assume that the supremum is not attained. The case where $\sup_{(-\infty, 0)} L = \lim_{t \rightarrow 0^-} L(t)$ can be studied similarly to the above. 

Therefore let us study the case where $\sup_{(-\infty, 0)} L = \lim_{t \rightarrow -\infty} L(t)$.
We can assume that there is an interval $I \coloneqq (-\infty, \tau) \subseteq (-\infty, 0)$ such that $L|_I \ge \sup L -\varepsilon$ and such that there are no local maxima  and no local minima in $I$. Namely, otherwise we could proceed as we did above. Note however that for the case of local minima we have to consider intervals of the kind $(t_1, t_1 + \delta)$ instead. 

Having no local extrema implies together with continuity that the function $L$ is monotone on $I$. Since $\sup_{(-\infty, 0)} L = \lim_{t \rightarrow -\infty} L(t)$, it must be monotonically decreasing and thus satisfy $L' \le 0$ at all points of differentiability in $I$, so Lebesgue-almost everywhere. 
Moreover 
\beq
\int_I L' = \int_{-\infty}^\tau L'(t) \, dt = L(\tau) - \sup L \ge - \varepsilon.
\eeq
Therefore $L'|_I$ is summable. There also exists a differentiability point $t_0$ such that $L'(t_0) \ge -\varepsilon$, otherwise we would get a contradiction with summability from $L'(t_0) < -\varepsilon$ a.e in $I$.
\end{proof}

\end{document}